\documentclass[11pt]{article}
\usepackage{amsmath,amsfonts,amssymb,amsthm}
\usepackage{geometry}
\usepackage{setspace}
\usepackage{amsthm}
\usepackage{hyperref}
\usepackage{xcolor}
\usepackage{graphicx}
\usepackage{indentfirst}
\newtheorem{definition}{Definition}[section]
\newtheorem{lemma}{Lemma}
\newtheorem{theorem}{Theorem}
\newtheorem{remark}{Remark}

\begin{document}

% Заголовок по центру как в Word
\title{A modified Brinkman penalization fictitious domain method for the unsteady Navier-Stokes equations}

\author{
	\and Zhanybek Baitulenov\thanks{Faculty of Mechanics and Mathematics, Al-Farabi Kazakh National University, Almaty 050040, Kazakhstan, janibekbb47@gmail.com.} \and Maxim Olshanskii \thanks{Department of Mathematics, University of Houston, 651 PGH, Houston, 77204, TX, USA, 
		maolshanskiy@uh.edu. } \and Almas Temirbekov\thanks{Faculty of Mechanics and Mathematics, Al-Farabi Kazakh National University, Almaty 050040, Kazakhstan, almastemir86@gmail.com.} \and Nurlan Temirbekov\thanks{Faculty of Mechanics and Mathematics, Al-Farabi Kazakh National University, Almaty 050040, Kazakhstan, ntemirbekov314@gmail.com.} \and Syrym Kasenov\thanks{Faculty of Mechanics and Mathematics, Al-Farabi Kazakh National University, Almaty 050040, Kazakhstan, syrym.kasenov@gmail.com.} 
}

\maketitle

\begin{abstract} This paper investigates a modification of the fictitious domain method with continuation in the lower-order coefficients for the unsteady Navier-Stokes equations governing the motion of an incompressible homogeneous fluid in a bounded 2D or 3D domain. The modification enables {a solution-dependent} choice of the critical parameter. Global-in-time existence and convergence of a weak solution to the auxiliary problem are proved, and local-in-time existence and convergence of a unique strong solution are established. For the strong solution, a new higher-order convergence rate estimate in the penalization parameter is obtained. The introduced framework  allows us to apply a pointwise divergence free finite element method as a discretization technique, leading to strongly mass conservative discrete  fictitious domain method. A numerical example illustrates the performance of the method.
\end{abstract}

\section{Introduction}

The numerical solution of boundary value problems in complex domains remains a central challenge in scientific computing. Several approaches have been developed over the past decades to address geometric complexity, including curvilinear grid methods and fictitious domain formulations. While curvilinear grids require intricate mesh generation and coordinate transformations, geomtrically unfitted methods such as the fictitious domain method (FDM), immersed boundary and interface methods, as well as extended, cut and trace finite element techniques, have become increasingly popular; see, e.g., \cite{Glowinski1994,peskin2002immersed,leveque1994immersed,belytschko1999elastic,burman2025cut,olshanskii2018trace,Lehrenfeld2016}.

The fictitious domain method reformulates the original problem posed on a complex physical domain into one defined on a simple background domain that fully embeds  it. This reformulation permits the use of structured meshes and efficient solvers. The approach emerged in 1980's and was extensively developed by Glowinski and coauthors in a series of  papers \cite{Glowinski1994,Glowinski1995,GlowinskiPan2000,Glowinski2001,Girault2001}. In parallel, related operator-splitting and decomposition strategies were advanced by Marchuk and collaborators \cite{Marchuk1986,agoshkov1988poincare}. These works provided the theoretical and algorithmic foundations for modern fictitious domain formulations.

Since then, the method has been widely applied to incompressible fluid dynamics. For instance, Smagulov et al.~\cite{Smagulov2000} proposed a fictitious domain formulation for the Navier--Stokes equations with an efficient pressure treatment based on a Poisson problem. The method has been applied to simulate particulate flows \cite{Wu2020}, non-Newtonian fluids \cite{He2017}, two-phase and interface Stokes flows\cite{Court2019,corti2024low}, and fluid--structure interaction problems involving moving or deformable bodies \cite{boffi2017fictitious,Wang2017,boilevin2019loosely,Wang2019}. Recent developments include penalty-based stabilization techniques \cite{Zhou2018}, distributed Lagrange multiplier approaches for interface problems \cite{WangSun2017}, and phase-field formulations for particulate flows in evolving geometries \cite{Reder2021}. A review of a recently introduced unfitted finite element fictitious domain method for fluid problems based on a distributed Lagrange multiplier  can be found in \cite{alshehri2025advances}.
The theoretical foundation of fictitious domain methods involves establishing well-posedness of the extended problem and deriving convergence estimates for the approximate solution restricted to the physical domain. 

In this paper we develop and analyze a modification of the  fictitious domain method with continuation in the lower-order coefficients (also known as \textit{Brinkman penalization}) for the unsteady incompressible Navier--Stokes equations in three dimensions. We establish the global-in-time existence and convergence of weak solutions to the auxiliary problem, and the local-in-time existence and convergence of unique strong solutions. For the strong solution, we derive a new convergence rate estimate that compares favorably with existing results for the standard FDM.

We further consider an unfitted finite element method that employs a background mesh defined on an ambient domain $D$, which is independent of the geometry of the physical domain $\Omega$.
{The resulting formulation is particularly well suited for the use of strongly divergence-free finite elements, the incorporation of which in more elaborate unfitted methods, such as CutFEM, entails substantially higher implementation complexity~\cite{neilan2025unfitted}}.
In our implementation, the mesh is barycentrically refined to enable the use of an inf–sup stable, divergence-free Scott–Vogelius finite element method. This leads to a discrete fictitious domain method that is strongly mass conservative. A series of numerical experiments with a synthetic solution illustrate the theory and assess the performance of the method.

The remainder of the paper is organized as follows. In Section~\ref{S2}, we formulate the problem and introduce the fictitious domain method. The global existence of a weak solution and the local existence and uniqueness of a strong solution to the fictitious domain problem are established in Section~\ref{S3}. The convergence estimate in terms of the method's critical parameter is presented in Section~\ref{S5}. Section~\ref{S6} introduces the strongly conservative finite element fictitious domain method and presents results from several numerical experiments.

\section{Problem Statement}\label{S2}

We are interested in the Navier--Stokes system describing the motion of a homogeneous, unsteady, incompressible viscous fluid. Thus, in a bounded domain 
\[
Q_T = (0,T)\times\Omega, \qquad \Omega \subset \mathbb{R}^3, \quad S = \partial\Omega,
\] 
we consider the nonlinear unsteady Navier--Stokes equations for an incompressible homogeneous fluid:
\begin{equation} \label{St 1}
	\frac{\partial v}{\partial t} + (v\cdot\nabla)v = \mu\Delta v - \nabla P + f,
\end{equation}
\begin{equation} \label{St 2}
	\operatorname{div} v=0,
\end{equation}
with initial and boundary conditions
\begin{equation} \label{St 3}
	v|_{t=0} = v_0(x),\qquad v|_{S} = 0.
\end{equation}
Here $v$ is the fluid velocity vector field, $P$ the pressure, $f$ the body force vector field, and $\mu>0$ the viscosity coefficient. 

In \cite{angot1999penalization,carbou2003boundary}, the so-called $L^2$-penalization or Brinkman penalization version of the fictitious domain method was studied for \eqref{St 1}–\eqref{St 3}; {see also \cite{angot1999analysis,kevlahan2001computation,sharaborin2021coupled,kevlahan2005adaptive,liu2007brinkman} among other contributions}. The approach has recently gained popularity in fluid topology optimization and related applications; see, e.g., \cite{abdelhamid2023calculation,boscolo2025scaling} and references therein. 

Let $D \supset \Omega$ be a domain with boundary $S_1 = \partial D$. The $L^2$-penalization method can be formulated as the following auxiliary problem:
\begin{equation}\label{St 4}
	\frac{\partial v^\varepsilon}{\partial t} + (v^\varepsilon \cdot \nabla)v^\varepsilon = \mu\Delta v^\varepsilon - \nabla P^\varepsilon - \frac{\xi(x)v^\varepsilon}{\varepsilon} + f^\varepsilon,
\end{equation}
\begin{equation} \label{St 5}
	\operatorname{div} v^\varepsilon = 0,
\end{equation}
\begin{equation}\label{St 6}
	v^\varepsilon|_{t=0} = v_0(x),\qquad v^\varepsilon|_{S_1} = 0,
\end{equation}
where $\varepsilon > 0$ is a small parameter, 
\[
\xi(x) = \begin{cases} 
	0, & x \in \Omega, \\[4pt]
	1, & x \in D_1 = D \setminus \Omega, 
\end{cases}
\qquad 
f^{\varepsilon} = \begin{cases} 
	f, & x \in \Omega, \\[4pt]
	0, & x \in D_1,
\end{cases}
\]
and $v_0(x)$ is extended as zero outside $\Omega$. An error estimate for this method was established in \cite{angot1999penalization,carbou2003boundary} in terms of powers of $\varepsilon$:
\begin{equation}\label{St 7}
	\| v^\varepsilon - v \|_{L_2(0,T;L_2(\Omega))} \leq C\sqrt{\varepsilon},
\end{equation}
where the constant $C$ depends on the solution $v$ of \eqref{St 1}--\eqref{St 3}. {Some authors (see, e.g., \cite{kevlahan2005adaptive,sharaborin2021coupled}) have reported that the convergence rates observed in practice with respect to $\varepsilon$ can be higher.} 

In this work we propose a modification of the fictitious domain method, in which the critical ``extension'' function $\xi(x)/\varepsilon$ is replaced by a solution-dependent factor. The modified method includes \eqref{St 4}--\eqref{St 6} as a special case. We prove a convergence estimate for the modified method, which is new and  generalizes \eqref{St 7}. Furthermore, it improves \eqref{St 7} also for the case $\beta=0$.

Specifically, we consider the problem
\begin{equation} \label{St 8}
	\frac{\partial v^\varepsilon}{\partial t} + (v^\varepsilon \cdot \nabla)v^\varepsilon = \mu\Delta v^\varepsilon - \nabla P^\varepsilon - \frac{\xi(x)v^\varepsilon}{\varepsilon\|v^\varepsilon\|_{L_2(D_1)}^{\beta}} + f^{\varepsilon},
\end{equation}
\begin{equation} \label{St 9}
	\operatorname{div} v^\varepsilon = 0, \qquad 0 \le \beta < 1,
\end{equation}
\begin{equation} \label{St 10}
	v^\varepsilon|_{t=0} = v_0(x),\qquad v^\varepsilon|_{S_1} = 0.
\end{equation}
Clearly, for $\beta = 0$ the problem \eqref{St 8}--\eqref{St 10} reduces to \eqref{St 4}--\eqref{St 6}. 

\begin{remark}{
	If the problem \eqref{St 1}--\eqref{St 3} is supplied with a non-homogeneous Dirichlet boundary condition  $v=v_b$ on $S$, then the fictitious domain method is modified to account for it as follows. Consider arbitrary smooth $v_1:\,D\to \mathbb{R}^d$ such that $(v_1)|_{S}=v_b$ and update the penalty term in the momentum equation \eqref{St 8} to $\frac{\xi(x)(v^\varepsilon-v_1)}{\varepsilon\|v^\varepsilon-v_1\|_{L_2(D_1)}^{\beta}}$.}
\end{remark}

In this paper we establish the global existence of a weak solution to \eqref{St 8}--\eqref{St 10}. We also prove local existence of a unique strong solution and derive an estimate for its convergence to the solution of \eqref{St 1}--\eqref{St 3}.

We define $V(D)$ as the closure in the $W_{2}^{1}(D)$-norm of the set of infinitely differentiable, solenoidal, compactly supported vector functions in $D$. 

\begin{definition}
	A weak solution of the auxiliary problem \eqref{St 8}--\eqref{St 10} is a function $v^\varepsilon \in L_2(0,T;V(D)) \cap L_\infty(0,T;L_2(D))$ such that
	\begin{multline*}
	\int_0^T \int_D \left( -v^\varepsilon \Phi_t + \mu \nabla v^\varepsilon: \nabla\Phi - (v^\varepsilon \cdot \nabla)\Phi \cdot v^\varepsilon - f^\varepsilon \Phi \right) \, dx dt\\
	+ \int_0^T \int_{D_1} \frac{v^\varepsilon \cdot \Phi}{\varepsilon \|v^\varepsilon\|_{L_2(D_1)}^\beta} \, dx dt - \int_D v_0(x)\Phi(0,x)\,dx = 0,
	\end{multline*}
	for all  test functions $\Phi \in L_2(0,T;V(D))$ with $\Phi_t \in L_2(0,T;L_2(D))$ and $\Phi(T,\cdot)=0$.
\end{definition}

\begin{definition}
	A strong solution of the auxiliary problem \eqref{St 8}--\eqref{St 10} is a pair of functions $(v^\varepsilon, P^\varepsilon)$ that satisfy equations \eqref{St 8}--\eqref{St 9} and boundary conditions \eqref{St 10} almost everywhere.
\end{definition}

The  weak and strong solutions of the original problem \eqref{St 1}--\eqref{St 3} are defined similarly.

Next, we recall a few well-known definitions and results (see, e.g., {section~II.2 of} \cite{antontsev1989boundary}), which will be used later.

\begin{lemma}\label{L1}
	A bounded subset of a reflexive Banach space is weakly compact. Moreover, if \( u \) is the weak limit of a sequence \( \{u_n\} \), then
	\[
	\|u\| \leq \lim\limits_{n \to \infty} \|u_n\|.
	\]
\end{lemma}

\begin{lemma}
	A closed bounded ball in the space $L_\infty(\Omega)$ is compact in the sense of weak-* convergence.
\end{lemma}

\begin{lemma}\label{L3}
	Let $\Omega\subset\mathbb{R}^n$ be compact. A subset \( \mathcal{K}\subset C(\Omega) \) is compact if and only if it is closed, uniformly bounded, and equicontinuous, i.e.,
	\[
	\exists k_0 > 0 : \|u\|_{C(\Omega)} \leq k_0, \quad \forall u \in \mathcal{K};
	\]
	\[
	\forall \varepsilon > 0 \ \exists \delta = \delta(\varepsilon) > 0 : \forall x_1, x_2 \in \overline{\Omega},\ |x_1 - x_2| < \delta \ \Rightarrow \ |u(x_1) - u(x_2)| < \varepsilon \quad \forall u \in \mathcal{K}.
	\]
	
	\noindent
	Assume now that $\Omega\subset\mathbb{R}^n$ is bounded. For the compactness of a closed set \( K \) in the space \( L_p(\Omega) \), \( 1 < p < \infty \), it is necessary and sufficient that all functions in \( K \) are uniformly bounded and equicontinuous in the \( L_p(\Omega) \) norm, i.e.,
	\[
	\exists k_1 > 0 : \|u\|_{L_p(\Omega)} \leq k_1,\quad \forall u \in K;
	\]
	\[
	\forall \varepsilon > 0\ \exists \delta = \delta(\varepsilon) > 0 : \forall \Delta,\ |\Delta| < \delta,\ \forall u \in K \ \Rightarrow \ \| u_\Delta - u \|_{L_p(\Omega)} < \varepsilon,
	\]
	where \( u_\Delta(x) = u(x + \Delta) \), and \( u_\Delta(x) \equiv 0 \) when \( x + \Delta \notin \Omega \).
	
	In the space $C^\alpha(\Omega)$, $ 0 \leq \alpha < 1$, a closed set of functions is compact if the functions are uniformly bounded in the $C^\beta(\Omega)$-norm for some $\beta > \alpha$.
\end{lemma}

The following result can be found as eq. (2.23) in \cite{antontsev1989boundary}.
\begin{lemma}\label{Lem4}
	Let ${\Omega}\subset\mathbb{R}^n$ be bounded {and Lipschitz}, and $u \in W^2_1(\Omega)$, $\int_\Omega u =0$. Then
	\[
	\|u\|_{L_q(\partial\Omega)} \leq C \|\nabla u\|_{L_2(\Omega)}^{\alpha} \|u\|_{L_2(\Omega)}^{1-\alpha}, \quad \alpha\in (0,1), \quad q=\frac{2(n-1)}{n-2\alpha}.
	\]
\end{lemma}

\section{Weak and strong solutions to the fictitious domain problem}\label{S3}

We begin by proving the global existence of a weak solution to the auxiliary problem \eqref{St 8}--\eqref{St 10}.

\begin{theorem}\label{Th1}
	If \( f \in L_2(0,T;L_{6/5}(\Omega)) \) and \( v_0 \in L_2(\Omega) \), then there exists at least one weak solution to problem \textnormal{\eqref{St 8}--\eqref{St 10}} satisfying the estimate
	\begin{equation} \label{St 11}
		\|v^{\varepsilon}\|^2_{L_2(0,T;V(D))} + \|v^{\varepsilon}\|^2_{L_\infty(0,T;L_2(D))} 
		+ \frac{1}{\varepsilon} \int_0^T \|v^{\varepsilon}\|_{L_2(D_1)}^{2 - \beta} \, dt \leq C,
	\end{equation}
	{with some $C$ depending on $f$, $\mu$, $\Omega$, and $D$. }
	Furthermore, as $\varepsilon \to 0$, the weak solution of \eqref{St 8}--\eqref{St 10} converges to the weak solution of the original problem \eqref{St 1}--\eqref{St 3}.
\end{theorem}

\noindent
\textit{Proof.} To prove Theorem~\ref{Th1}, we employ the classical method of a priori estimates~\cite[{section~III.1}]{antontsev1989boundary}.  
Assuming for a moment that $v^{\varepsilon}$ is sufficiently smooth, we take the $L_2(D)$ scalar product of equation \eqref{St 8} with $v^{\varepsilon}$.  
After estimating the nonlinear terms using Hölder’s inequality, Sobolev embeddings, and Young’s inequality, we obtain
\begin{align*}
	\frac{1}{2}\frac{d}{dt}\|v^{\varepsilon}\|^{2}_{L_2(D)} 
	+ \mu\|\nabla v^{\varepsilon}\|^{2}_{L_2 (D)} 
	+ \frac{1}{\varepsilon} \|v^{\varepsilon}\|^{2-\beta}_{L_2 (D_1)} 
	&= \int_D f^{\varepsilon} v^{\varepsilon} \, dx  \\
	&\leq \|f^{\varepsilon}\|_{L_{6/5}(D)} \|v^{\varepsilon}\|_{L_6 (D)} \\
	&\leq C\|f\|_{L_{6/5}(\Omega)} \|\nabla v^{\varepsilon}\|_{L_2 (D)} \\
	&\leq \frac{\mu}{2} \|\nabla v^{\varepsilon}\|^{2}_{L_2 (D)} + C\|f\|^{2}_{L_{6/5} (\Omega)}.
\end{align*}
Integrating with respect to $t$ yields the estimate \eqref{St 11}.

\medskip
To proceed with the proof of the theorem, we next need a few auxiliary lemmas.

\begin{lemma}
	For any $\delta$, $0 < \delta < T$, the following inequality holds:
	\begin{equation} \label{St 12}
		\int_{0}^{T-\delta} \|v^{\varepsilon}(t+\delta) - v^{\varepsilon}(t)\|^{2}_{L_2 (D)} \, dt 
		\leq C(\delta^{1/2} + \delta^{1/(2 - \beta)}).
	\end{equation}
\end{lemma}

\begin{proof}
	The proof follows the approach in \cite{antontsev1989boundary}.  
	Fix $\delta$ and $t$, consider equation \eqref{St 8} on the interval $\tau \in [t, t + \delta]$, and test it with a function $\psi(\tau)$ in $L_2(D)$.  
	Integrating with respect to $\tau$ over $[t, t+\delta]$, and then choosing $\psi = v^\varepsilon(t+\delta) - v^\varepsilon(t)$, we integrate once more with respect to $t \in (0, T-\delta)$.  
	After estimating terms analogously to \cite[{section~III.1.3}]{antontsev1989boundary}, we are left only with an additional contribution coming from the nonlinear term over $D_1$.  
	Estimating this term yields
\begin{align*}
	\int\limits_0^{T-\delta} \int\limits_t^{t+\delta}& \frac{\int\limits_{D_1} v^\varepsilon \psi \, dx}{\varepsilon \|v^\varepsilon\|_{L_2(D_1)}^\beta} \, d\tau \, dt
	\leq \int\limits_0^{T-\delta} \int\limits_t^{t+\delta} \frac{1}{\varepsilon} \|v^\varepsilon\|_{L_2(D_1)}^{1-\beta} \|\psi\|_{L_2(D_1)} \, d\tau \, dt  \\
	&\leq \int\limits_0^{T-\delta} \varepsilon^{\frac{-1}{2-\beta}} \|\psi\|_{L_2(D_1)} 
	\int\limits_t^{t+\delta} \varepsilon^{\frac{-(1-\beta)}{2-\beta}} \|v^\varepsilon\|_{L_2(D_1)}^{1-\beta} \, d\tau \, dt \\
	&\leq \int\limits_0^{T-\delta} \varepsilon^{\frac{-1}{2-\beta}} \|\psi\|_{L_2(D_1)} 
	\left( \int\limits_t^{t+\delta} \left( \varepsilon^{\frac{-(1-\beta)}{2-\beta}} \|v^\varepsilon\|_{L_2(D_1)}^{1-\beta} \right)^{\frac{2-\beta}{1-\beta}} \, d\tau \right)^{\frac{1-\beta}{2-\beta}} \left(\int_{t}^{t+\delta} d\tau\right)^{\frac{1}{2-\beta}}dt \\
	&\leq \int_{0}^{T-\delta} \varepsilon^{-\frac{1}{2-\beta}} \| \psi \|_{L_2 (D_1)} \left( \int_{0}^{T} \varepsilon^{-1} \|v^{\varepsilon} \|^{2-\beta}_{L_2 (D_1)} d\xi \right)^{\frac{1-\beta}{2-\beta}} \delta^{\frac{1}{2-\beta}} dt\\ &\leq
	C \delta^{\frac{1}{2-\beta}} \int\limits_0^{T} \varepsilon^{-\frac{1}{2-\beta}} \|v^{\varepsilon}(t)\|_{L_2(D_1)} \, dt \\
	&\leq C \delta^{\frac{1}{2-\beta}} \left( \int\limits_0^{T} \left( \varepsilon^{-\frac{1}{2-\beta}} \|v^{\varepsilon}(t)\|_{L_2(D_1)} \right)^{2-\beta} dt \right)^{\frac{1}{2-\beta}} \left( \int\limits_0^{T} dt \right)^{\frac{1-\beta}{2-\beta}}\\
	&\leq C \delta^{\frac{1}{2-\beta}} \left( \int_{0}^{T} \varepsilon^{-1} \|v^{\varepsilon}(t) \|^{2-\beta}_{L_2 (D_1)} dt \right)^{\frac{1}{2-\beta}} \leq C \delta^{\frac{1}{2-\beta}},
\end{align*}
	where the a priori estimate \eqref{St 11} has been used.  
	This completes the proof.
\end{proof}

\medskip
The remainder of the proof of Theorem~\ref{Th1} is based on the Galerkin method.  
We seek an approximate solution to \eqref{St 8}--\eqref{St 10} of the form
\[
v_{N}^{\varepsilon}(t,x) = \sum_{m=1}^{N} \alpha_{m}(t)\omega_{m}(x),
\]
where the coefficients $\alpha_{m}(t)$ are determined from the finite-dimensional system
\begin{multline} \label{St 13}
%	\begin{aligned}
		\big( (v_{N}^\varepsilon )_{t}, \omega_m \big)_{L_2(D)} 
		+ \mu \big( \nabla v_{N}^\varepsilon , \nabla \omega_m \big)_{L_2(D)} 
		+ \big( (v_{N}^\varepsilon \cdot \nabla) v_{N}^\varepsilon, \omega_m \big)_{L_2(D)} \\
		+ \frac{1}{\varepsilon} \int_{D_1} 
		\left( \frac{v_{N}^\varepsilon}{\| v_{N}^\varepsilon \|_{L_2(D_1)}^\beta}, \omega_m \right) dx 
		= \big( f^\varepsilon, \omega_m \big)_{L_2(D)}, \quad m = 1, \dots, N.
%	\end{aligned}
\end{multline}
The initial condition is chosen as
\begin{equation} \label{St 14}
	v_{0N} = \sum_{m=1}^{N} \alpha_m(0) \omega_m(x), \qquad v_{0N} \to v_0 \quad \text{in } V(D) \text{ as } N \to \infty,
\end{equation}
where $\omega_m$ are the eigenfunctions of the Stokes problem
\[
\Delta \omega_m - \nabla p_m = \lambda_m \omega_m, 
\quad \nabla \cdot \omega_m = 0, 
\quad \omega_m|_{S_1} = 0, 
\quad m = 1, 2, \ldots
\]
forming an orthonormal basis in $L_2(D)$.

Multiplying \eqref{St 13} by $\alpha_m(t)$, summing over $m = 1,\dots,N$, and applying the estimates above gives \eqref{St 11} for $v_N^\varepsilon$.  
Similarly, estimate \eqref{St 12} is obtained.  
Using Lemmas \ref{L1}--\ref{L3}, the bounds \eqref{St 11}--\eqref{St 12} allow us to extract from $\{v_N^\varepsilon\}$ a subsequence such that:
\begin{itemize}
	\item $v_N^\varepsilon \rightharpoonup v^\varepsilon$ weakly-* in $L_\infty(0,T;L_2(D))$ and weakly in $L_2(0,T;V(D))$,
	\item $v_N^\varepsilon \to v^\varepsilon$ strongly in $L_2(0,T;L_2(D))$,
	\item $v_N^\varepsilon \rightharpoonup v^\varepsilon$ weakly in $L_{2-\beta}(0,T;L_2(D_1))$ as $N \to \infty$.
\end{itemize}

Standard arguments then show that the limit $v^\varepsilon$ is a weak solution of \eqref{St 8}--\eqref{St 10}, and satisfies estimates \eqref{St 11}--\eqref{St 12}.  
Passing to the limit as $\varepsilon \to 0$ along a subsequence, we further obtain:
\begin{itemize}
	\item $v^\varepsilon \rightharpoonup v$ weakly-* in $L_\infty(0,T;L_2(D))$ and weakly in $L_2(0,T;V(D))$,
	\item $v^\varepsilon \to v$ strongly in $L_2(0,T;L_2(D))$,
	\item $v^\varepsilon \to 0$ strongly in $L_{2-\beta}(0,T;L_2(D_1))$.
\end{itemize}

Hence, the limiting function $v(t,x)$ is a weak solution of the original problem \eqref{St 1}--\eqref{St 3}.  
This completes the proof of Theorem~\ref{Th1}.

We proceed in this section with proving the local existence of the strong solution to the auxiliary problem. 

\begin{theorem} Assume $f \in L_2(0, T; L_2(\Omega))$ and $v_0(x) \in V(\Omega)$, then for  a sufficiently small  $ T_0>0$, there exists the unique solution $ v^{\varepsilon}$ to \eqref{St 8}--\eqref{St 10} such that 
	$ v^{\varepsilon}\in L_2 (0,T_0; W^{2}_{2} (D))\cap L_\infty (0,T_0; V(D))$, $v^{\varepsilon}_{t}\in L_2 (0,T_0 ; L_2 (D))$ and   the following estimate holds: 
\begin{multline} \label{St 15}
\| v^{\varepsilon} \|^{2}_{L_\infty (0,T_0; V(D))} + \frac{1}{\varepsilon} \|v^{\varepsilon} \|^{2 - \beta}_{L_\infty (0,T_0; L_2 (D_1))} + \\
+ \|v^{\varepsilon}_{t} \|^{2}_{L_2 (0,T_0 ; L_2 (D))} 
+ \| v^{\varepsilon} \|^{2}_{L_2 (0,T_0; W^{2}_{2} (D))} 
+ \frac{1}{\varepsilon} \int_{0}^{T_0} \|\nabla v^{\varepsilon} \|^{2-\beta}_{L_2 (D_1)} \, dt 
\leq C,
\end{multline}
{with some $C$ depending on $f$, $\mu$, $\Omega$, and $D$.}
\end{theorem}

\begin{proof} % --- assumptions: smooth domain, div-free fields, homogeneous BCs for Stokes operator ---
	Denote by $\tilde\Delta $ the Stokes operator.  
Since $D$ has a smooth boundary, by the Stokes regularity we have
	\[
	\|v\|_{W^2_2(D)\cap V(D)} \le C_{\rm reg}\,\|\tilde\Delta v\|_{L_2(D)}
	\]
for $v\in W^2_2(D)\cap V(D)$.
	Hence, using interpolation and Sobolev embeddings,
	\[
	\|\nabla v\|_{L^3(D)} \le \|\nabla v\|_{L^2(D)}^{1/2}\|\nabla v\|_{L^6(D)}^{1/2}
	\le C\,\|\nabla v\|_{L^2(D)}^{1/2}\|v\|_{W^2_2(D)}^{1/2}
	\le C\,\|\tilde\Delta v\|_{L^2(D)}^{1/2}\|\nabla v\|_{L^2(D)}^{1/2}.
	\]
	
	Taking the $L_2(D)$-inner product of \eqref{St 8} with $v_t$  yields
	\[
	\|v_t\|_{L_2(D)}^2 - \frac{\mu}{2}\frac{d}{dt}\|\nabla v\|_{L_2(D)}^2
	+ \int_D (v\cdot\nabla)v\cdot v_t\,dx
	= \int_D f\,v_t\,dx - \frac{1}{\varepsilon\|v\|_{L_2(D_1)}^\beta}\int_{D_1} v\cdot v_t\,dx.
	\]
We estimate the nonlinear term by the Gagliardo--Nirenberg and Young inequalities:
	\[
	\begin{aligned}
		\Big|\int_D (v\cdot\nabla)v\cdot v_t\,dx\Big|
		&\le \|v_t\|_{L_2(D)}\|v\|_{L_6(D)}\|\nabla v\|_{L_3(D)} \\
		&\le C\|v_t\|_{L_2(D)}\|\nabla v\|_{L_2(D)}\|\tilde\Delta v\|_{L_2(D)}^{1/2}\|\nabla v\|_{L_2(D)}^{1/2} \\
		&\le \delta_1\|v_t\|_{L_2(D)}^2 + C_{\delta_1}\|\tilde\Delta v\|_{L_2(D)}\|\nabla v\|_{L_2(D)}^3 \\
		&\le \delta_1\|v_t\|_{L_2(D)}^2 + \theta_1\|\tilde\Delta v\|_{L_2(D)}^2 + C_{\theta_1,{\delta_1}}\|\nabla v\|_{L_2(D)}^6.
	\end{aligned}
	\]
	Similarly,
	\[
	\int_D f v_t \,dx \le \delta_2\|v_t\|_{L_2(D)}^2 + C_{\delta_2}\|f\|_{L_2(D)}^2.
	\]
	The damping-derivative term is exactly
	\[
	\frac{1}{\varepsilon\|v\|_{L_2(D_1)}^\beta}\int_{D_1} v\cdot v_t\,dx
	= \frac{1}{2\varepsilon(2-\beta)}\frac{d}{dt}\|v\|_{L_2(D_1)}^{2-\beta}.
	\]
	Choosing $\delta_1+\delta_2<1$ we obtain
	\begin{equation}\label{St 17}
		\|v_t\|_{L_2(D)}^2 + \frac{d}{dt}\Big({\frac\mu2}\|\nabla v\|_{L_2(D)}^2 + \frac{1}{\varepsilon}\|v\|_{L_2(D_1)}^{2-\beta}\Big)
		\le \theta_1\|\tilde\Delta v\|_{L_2(D)}^2 + C\big(\|\nabla v\|_{L_2(D)}^6 + \|f\|_{L_2(D)}^2\big).
	\end{equation}
	Next we take the $L_2(D)$-inner product of \eqref{St 8} with $\tilde\Delta v$.
 Then
	\[
	{\mu}\|\tilde\Delta v\|_{L_2(D)}^2 - \frac{1}{\varepsilon\|v\|_{L_2(D_1)}^\beta}\int_{D_1} v\cdot\tilde\Delta v\,dx
	= \int_D (v_t+(v\cdot\nabla)v-f)\cdot\tilde\Delta v\,dx.
	\]
	Estimating the right-hand side by Young's inequality gets:
	\[
	\begin{aligned}
		\int_D v_t\cdot\tilde\Delta v &\le \theta_2\|\tilde\Delta v\|_{L_2}^2 + C_{\theta_2}\|v_t\|_{L_2}^2,\\
		\int_D (v\cdot\nabla)v\cdot\tilde\Delta v &\le \theta_3\|\tilde\Delta v\|_{L_2}^2 + C_{\theta_3}\|\nabla v\|_{L_2}^6,\\
		\int_D f\cdot\tilde\Delta v &\le \theta_4\|\tilde\Delta v\|_{L_2}^2 + C_{\theta_4}\|f\|_{L_2}^2.
	\end{aligned}
	\]
	For the damping term on the left, we integrate by parts on $D_1$ (using $\operatorname{div}v=0$ and homogeneous boundary conditions) to show
	\[
	-\frac{1}{\varepsilon\|v\|_{L_2(D_1)}^\beta}\int_{D_1} v\cdot\tilde\Delta v\,dx
	= \frac{\|\nabla v\|_{L_2(D_1)}^2}{\varepsilon\|v\|_{L_2(D_1)}^\beta}.
	\]
	By the Poincare inequality on $D_1$ there is $C_P>0$ with $\|v\|_{L_2(D_1)}\le C_P\|\nabla v\|_{L_2(D_1)}$, hence
	\[
	\frac{\|\nabla v\|_{L_2(D_1)}^2}{\varepsilon\|v\|_{L_2(D_1)}^\beta}
	\ge \frac{1}{\varepsilon C_P^\beta}\|\nabla v\|_{L_2(D_1)}^{2-\beta}.
	\]
	Collecting estimates and ensuring $\theta_2+\theta_3+\theta_4<1$ yields
	\begin{equation}\label{St 19}
		{\mu}\|\tilde\Delta v\|_{L_2(D)}^2 + \frac{1}{\varepsilon}\|\nabla v\|_{L_2(D_1)}^{2-\beta}
		\le C\big(\|v_t\|_{L_2(D)}^2 + \|\nabla v\|_{L_2(D)}^6 + \|f\|_{L_2(D)}^2\big),
	\end{equation}
	where $C$ depends on the choice of the small parameters and on $C_{\rm reg}$, $C_P$, etc.
	
Adding inequality \eqref{St 19} to \eqref{St 17} and dividing by $\theta_1>0$, we obtain  
\[
\| v^{\varepsilon}_{t} \|^{2}_{L_2(D)} + \frac{1}{\varepsilon} \|\nabla v^{\varepsilon}\|^{2-\beta}_{L_2(D_1)} + \frac{d}{dt}\Big( {\mu}\|\nabla v^{\varepsilon} \|^{2}_{L_2(D)} + \tfrac{1}{\varepsilon}\|v^{\varepsilon}\|^{2-\beta}_{L_2 (D_1)}\Big) 
\leq C\big( \|\nabla v^{\varepsilon}\|^{6}_{L_2 (D)} + \|f\|^{2}_{L_2 (\Omega)}\big).
\]
Setting  
\[
y(t)= {\mu}\|\nabla v^{\varepsilon} \|^{2}_{L_2(D)} + \tfrac{1}{\varepsilon}\|v^{\varepsilon}\|^{2-\beta}_{L_2 (D_1)},
\]
we get
\[
\frac{dy}{dt}\leq C \left(y^3(t)+\|f\|^{2}_{L_2 (\Omega)} \right),
\]
{with $C$ depending also on $\mu$.}
Integrating over $(0,t)$ gives  
\[
y(t)-y(0)\leq C\int_0^t y^3(\tau)\,d\tau + C\|f\|^{2}_{L_2 (Q_T)}.
\]
Let  
\[
z(t)=\int_0^t y^3(\tau)\,d\tau, \qquad B=C\|f\|^{2}_{L_2(Q_T)}.
\]
Then $y(t)\le y(0)+Cz(t)+B$ and hence  
\[
z'(t)=y(t)^3 \le (y(0)+Cz(t)+B)^3.
\]
Defining $R(t)=y(0)+Cz(t)+B$, we obtain
\[
R'(t)=Cz'(t)\le C R^3(t), \qquad R(0)=y(0)+B.
\]
By comparison with the solution of $S'(t)=C S^3(t)$, $S(0)=R(0)$, we find
\[
R(t)\le \left[-2Ct+(y(0)+B)^{-2}\right]^{-1/2}.
\]
Thus, choosing $T_0>0$ such that $-2CT_0+(y(0)+B)^{-2}>0$, we deduce that $R(t)$, and therefore $y(t)$, remain bounded on $[0,T_0]$. This yields the estimate
\begin{equation} \label{St 20}
	\|\nabla v^{\varepsilon} \|^2_{L_\infty(0,T_0;L_2(D))} + \tfrac{1}{\varepsilon} \| v^{\varepsilon} \|^{2-\beta}_{L_\infty(0,T_0;L_2(D_1))} + \| v^{\varepsilon}_{t} \|^2_{L_2(0, T_0; L_2(D))} + \tfrac{1}{\varepsilon} \int_0^{T_0} \|\nabla v^{\varepsilon} \|^{2-\beta}_{L_2(D_1)}\,dt \leq C.
\end{equation}
Finally, from \eqref{St 19} it also follows that
\begin{equation} \label{St 21}
	\|\tilde{\Delta} v^{\varepsilon} \|^2_{L_2(0,T_0;L_2(D))} \leq C
\end{equation}
for some $T_0>0$
depending on the solution of the homogeneous equation \eqref{St 1}--\eqref{St 3}, where the value \( T_0 \) is determined by the given data of the original problem \eqref{St 1}--\eqref{St 3}. 

Thus, we obtain estimate \eqref{St 15}. The  further proof of the local existence of a strong solution to the auxiliary problem \eqref{St 8}--\eqref{St 10} % and its convergence to the strong solution of the original problem \eqref{St 1}--\eqref{St 3} as \(\varepsilon \to 0\) 
is carried out based on estimates  \eqref{St 15} similarly to \cite[{section~III.3}]{antontsev1989boundary} and does not present any particular difficulty.
\medskip

\textbf{Proof of uniqueness of the strong solution.} Assume there exist two strong solutions $(v_1^\varepsilon,P_1^\varepsilon)$ and $(v_2^\varepsilon,P_2^\varepsilon)$ and set
$\omega=v_1^\varepsilon-v_2^\varepsilon$, $\pi=P_1^\varepsilon-P_2^\varepsilon$. Then from \eqref{St 8}--\eqref{St 10} we get
\begin{equation}\label{St22corr}
	\omega_t+(\omega\cdot\nabla)v_1^\varepsilon+(v_2^\varepsilon\cdot\nabla)\omega=\mu\Delta\omega-\nabla\pi
	-\frac{\xi(x)}{\varepsilon}\Big(\frac{v_1^\varepsilon}{\|v_1^\varepsilon\|_{L_2(D_1)}^\beta}-\frac{v_2^\varepsilon}{\|v_2^\varepsilon\|_{L_2(D_1)}^\beta}\Big),
\end{equation}
with $\operatorname{div}\omega=0$, $\omega|_{t=0}=0$, $\omega|_{S_1}=0$. Multiply \eqref{St22corr} by $\omega$ and integrate over $D$ to obtain
\[
\frac12\frac{d}{dt}\|\omega\|_{L_2(D)}^2+\mu\|\nabla\omega\|_{L_2(D)}^2
+\frac{1}{\varepsilon}\int_{D_1}\Big(\frac{v_1^\varepsilon}{\|v_1^\varepsilon\|^\beta}-\frac{v_2^\varepsilon}{\|v_2^\varepsilon\|^\beta}\Big)\!\cdot\omega\,dx
=-\int_D(\omega\cdot\nabla)v_1^\varepsilon\cdot\omega\,dx.
\]
The convection term is estimated by the Gagliardo--Nirenberg and Young inequalities:
\[
\Big|\int_D(\omega\cdot\nabla)v_1^\varepsilon\cdot\omega\,dx\Big|
\le C\|\nabla v_1^\varepsilon\|_{L_2}\|\omega\|_{L_2}^{1/2}\|\nabla\omega\|_{L_2}^{3/2}
\le\frac\mu2\|\nabla\omega\|_{L_2}^2 + C\|\omega\|_{L_2}^2.
\]
For the damping term set $a=\|v_1^\varepsilon\|_{L_2(D_1)}$, $b=\|v_2^\varepsilon\|_{L_2(D_1)}$. Using Cauchy--Schwarz,
$\int_{D_1}v_1^\varepsilon\cdot v_2^\varepsilon\,dx\le ab$, we obtain the monotonicity estimate
\[
\frac{1}{\varepsilon}\int_{D_1}\Big(\frac{v_1^\varepsilon}{a^\beta}-\frac{v_2^\varepsilon}{b^\beta}\Big)\cdot\omega\,dx
\ge\frac{1}{\varepsilon}(a-b)(a^{1-\beta}-b^{1-\beta})\ge0.
\]
Combining the estimates yields
\[
\frac{d}{dt}\|\omega\|_{L_2(D)}^2 \le C\|\omega\|_{L_2(D)}^2,
\]
and since $\omega(0)=0$ the Gronwall lemma gives $\omega\equiv0$ on $[0,T]$. Finally, from \eqref{St22corr} and $\omega\equiv0$ we infer $\nabla\pi=0$ (indeed the right-hand side of \eqref{St22corr} vanishes in $W^{-1}_2$), so the pressure difference is constant. Hence the strong solution is unique.
   Thus, the theorem is proved.
\end{proof}

\section{Convergence of the strong solutions} \label{S5}

\begin{theorem}{Assume  the strong solution $v$ of  \eqref{St 1}--\eqref{St 3} belongs to $L_\infty(0,T;W^2_2(\Omega))$.} Then the strong solution $v^{\varepsilon}$ of the initial boundary value problem \eqref{St 8}--\eqref{St 10} as $\varepsilon \rightarrow 0$ converges to the strong solution of the limiting problem with the following estimate:
\begin{equation} \label{St 26}
    \left\| v^{\varepsilon} - v \right\|_{L_\infty(0,T;L_2(\Omega))}^2 + \left\| v^{\varepsilon} - v \right\|_{L_2(0,T;V(\Omega))}^2 \leq C \varepsilon^{A(k, \beta)},  
\end{equation} 
where $A(k, \beta) = \frac{{16-6k}}{16+12k-16\beta-3\beta k},\, 0 \leq \beta < 1, \, k\in(0,8/3)$, and a constant $C$ may depend on $k,\beta$, {on $\mu$} and various norms of $f$ and $v_0$, but not on $\varepsilon$.
\end{theorem}

\begin{remark}\label{Rem2}\rm 
{The limiting case of $k\to 0$ gives $A(0, \beta) = \frac{1}{1-\beta}$. 
Thus for $\beta=0$ we recover the `classical' estimate  \eqref{St 7}, however using simple energy arguments rather than asymptotic analysis of the boundary layer as in \cite{carbou2003boundary}.   For  
 $\beta \rightarrow 1,\ k\rightarrow 0$ we have  $A(k, \beta) \rightarrow \infty$, while the constant $C$ in \eqref{St 26} blows up for $\beta \rightarrow 1,\ k\rightarrow 0$.   }
%Furthermore, for $\beta=0$ we recover the `classical' estimate  \eqref{St 7} for $k=2/3$, while letting $ k\to 0$ improves it to  
%$\left\| v^{\varepsilon} - v \right\|_{L_2(0,T;V(\Omega))}\le C({s}) \varepsilon^{1-{s}}$ with any ${s}>0$ {and $C(s)$ blowing up for $s\to0$.} 
\end{remark}

\begin{proof}
	Let $\psi\in V(D)$ be a solenoidal test function with $\psi|_{S_1}=0$. Multiplying \eqref{St 1} by $\psi$ and integrating over $\Omega$, we obtain
	\begin{equation} \label{St 27-new}
		\left( \frac{\partial v}{\partial t}, \psi \right)_{L_2(\Omega)} 
		+ \mu \left( \nabla v, \nabla \psi \right)_{L_2(\Omega)}
		- \mu \int_{S} \frac{\partial v}{\partial n} \, \psi \, dS
		- \left( (v \cdot \nabla)\psi, v \right)_{L_2(\Omega)} 
		= (f, \psi)_{L_2(\Omega)},
	\end{equation}
	where $v,f$ are extended by zero into $D_1$.
	
	Similarly, multiplying \eqref{St 8} by $\psi$ and integrating over $D$ gives
	\begin{equation} \label{St 28-new}
		\left( \frac{\partial v^{\varepsilon}}{\partial t}, \psi \right)_{L_2(D)} 
		+ \mu \left( \nabla v^{\varepsilon}, \nabla \psi \right)_{L_2(D)}
		- \left( (v^{\varepsilon} \cdot \nabla)\psi, v^{\varepsilon} \right)_{L_2(D)}
		+ \frac{1}{\varepsilon} \int_{D_1} \frac{v^{\varepsilon}\cdot \psi}{\| v^{\varepsilon} \|^\beta_{L_2(D_1)}} \, dx
		= (f^{\varepsilon}, \psi)_{L_2(D)}.
	\end{equation}
		Subtracting \eqref{St 27-new} from \eqref{St 28-new}, and denoting $\omega = v^{\varepsilon} - v$, we find
	\begin{multline*}
	\left( \frac{\partial \omega}{\partial t}, \psi \right)_{L_2(D)} 
	+ \mu(\nabla \omega, \nabla \psi)_{L_2(D)}
	+ \mu \int_{S} \frac{\partial v}{\partial n}\,\psi\, dS \\
	+ \left( (\omega \cdot \nabla)\psi, v \right)_{L_2(D)}
	+ \left( (v^{\varepsilon} \cdot \nabla)\psi, \omega \right)_{L_2(D)}
	+ \frac{1}{\varepsilon} \int_{D_1} \frac{v^{\varepsilon} \cdot \psi}{\|v^{\varepsilon}\|_{L_2(D_1)}^{\beta}} dx = 0.
	\end{multline*}
	
	Choosing $\psi = \omega$ and noting that $\omega=v^{\varepsilon}$ in $D_1$, we obtain the energy identity
	\[
	\frac{1}{2} \frac{d}{dt} \|\omega\|_{L_2(D)}^2 
	+ \mu \|\nabla \omega\|_{L_2(D)}^2
	+ \frac{1}{\varepsilon} \|\omega\|^{2-\beta}_{L_2(D_1)}
	+ \mu \int_S \frac{\partial v}{\partial n}\,\omega \, dS
	- \big( (\omega \cdot \nabla)\omega, v \big)_{L_2(D)} = 0.
	\]
	Hence,
	\begin{equation} \label{St 29-new}
		\frac{1}{2} \frac{d}{dt} \|\omega\|_{L_2(D)}^2 
		+ \mu \|\nabla \omega\|_{L_2(D)}^2
		+ \frac{1}{\varepsilon} \|\omega\|^{2-\beta}_{L_2(D_1)}
		\leq \mu \left| \int_S \frac{\partial v}{\partial n}\,\omega \, dS \right|
		+ \left| \big( (\omega \cdot \nabla)\omega, v \big)_{L_2(D)} \right|.
	\end{equation}
	
	\medskip
	First we estimate  the boundary term. 
	By Hölder’s inequality, {the trace inequality and the regularity assumption for $v$, we get for any $k\ge0$}:
	\[
	\mu \left| \int_S \frac{\partial v}{\partial n}\,\omega \, dS \right|
	\leq \mu \left\| \frac{\partial v}{\partial n} \right\|_{L_{{\frac{4+3k}{1+3k}}}(S)} \|\omega\|_{L_{{\frac43+k}}(S)}{
	\le C \|v\|_{W^2_2(\Omega)} \|\omega\|_{L_{\frac43+k}(S)} \le 	C \|\omega\|_{L_{\frac43+k}(S)}}.
	\]
	{Applying the interpolation estimate from Lemma~\ref{Lem4} and restricting to  $k\in(0,\frac83)$} yields  
	\[
	\begin{split}
	{ \|\omega\|_{L_{\frac{4}{3}+k}(S)}}& 
	\leq C_2 \|\nabla \omega\|^{\frac{9k}{8+6k}}_{L_2(D_1)} \|\omega\|^{\frac{8-3k}{8+6k}}_{L_2(D_1)}\\
	&= C_2 \left(\left\| \nabla\omega \right\|_{L_2(D_1)} \left\| \omega \right\|^{(2-\beta)/2}_{L_2(D_1)}\right)^{\frac{9k}{8+6k}} \| \omega \|^{\frac{16-24k+9\beta k}{2(8+6k)}}_{L_2 (D_1)}  \\
	&\leq C_2 \left( \frac{\sqrt{\varepsilon}}{C_3} \left\| \nabla\omega \right\|^2_{L_2(D_1)} + \frac{C_3}{4\sqrt{\varepsilon}} \left\| \omega \right\|^{2-\beta}_{L_2(D_1)} \right)^{\frac{9k}{8+6k}} \left\| \omega \right\|^{\frac{16-24k+9\beta k}{2(8+6k)}}_{L_2(D_1)}\quad{\text{\small $\forall\, C_3>0$}}\\
	&= \left( \left\| \nabla\omega \right\|^2_{L_2(D_1)} + \frac{C_4}{\varepsilon} \left\| \omega \right\|^{2-\beta}_{L_2(D_1)} \right)^{\frac{9k}{8+6k}} \varepsilon^{\frac{9k}{16+12k}} \left\| \omega \right\|^{\frac{16-24k+9\beta k}{16+12k}}_{L_2(D_1)} \quad{\text{\small we let $C_3=C_2^{-\frac{8+6k}{9k}}$}} \\
	&\leq \delta \left(  \| \nabla\omega \|^2_{L_2(D_1)} + \frac{C_4}{\varepsilon} \left\| \omega \right\|^{2-\beta}_{L_2(D_1)} \right) + C_{\delta} \left(\varepsilon^{\frac{9k}{16+12k}} \left\| \omega \right\|^{\frac{16-24k+9\beta k}{16+12k}}_{L_2(D_1)}\right)^{\frac{8+6k}{8-3k}}\\
	&=\delta \left(  \| \nabla\omega \|^2_{L_2(D_1)} + \frac{C_4}{\varepsilon} \left\| \omega \right\|^{2-\beta}_{L_2(D_1)} \right) + C_{\delta} \left(\varepsilon^{\frac{9k}{2(8-3k)}} \left\| \omega \right\|^{\frac{16-24k+9\beta k}{2(8-3k)}}_{L_2(D_1)}\right),
	\end{split}
	\]
	with arbitrary $\delta>0$.
	
	\medskip
	Next we estimate of the nonlinear term.
	Using Sobolev embedding {and the regularity assumption for $v$, we have  $\|v\|_{L_6(D)} \le C\|\nabla v\|_{L_2(D)}\le C$. Thanks to this} and the Gagliardo–Nirenberg inequality $\|\omega\|_{L_3(D)} \le \|\omega\|_{L_2(D)}^{1/2}\|\nabla \omega\|_{L_2(D)}^{1/2}$, we obtain
	\[
	\big| ((\omega\cdot\nabla)\omega, v)_{L_2(D)} \big|
	\le C \|\nabla \omega\|_{L_2(D)}^{3/2}\|\omega\|_{L_2(D)}^{1/2}
	\le \frac{\mu}{2}\|\nabla \omega\|_{L_2(D)}^2 + C \|\omega\|_{L_2(D)}^2.
	\]
	
	\medskip
	We now  collect the  estimates and balance exponents.  
	Substituting the above into \eqref{St 29-new}, {and choosing $\delta>0$ small enough but independent of $\varepsilon$}, we obtain
	\begin{equation} \label{St 30-new}
		\frac{1}{2}\frac{d}{dt}\|\omega\|_{L_2(D)}^2
		+ C_5\left( \|\nabla \omega\|_{L_2(D)}^2 + \frac{1}{\varepsilon}\|\omega\|^{2-\beta}_{L_2(D_1)} \right)
		\le C_6 \|\omega\|_{L_2(D)}^2
		+ C_7 \varepsilon^{\frac{9k}{2(8-3k)}} \|\omega\|_{L_2(D_1)}^{\frac{16-24k+9\beta k}{2(8-3k)}}.
	\end{equation}
	Applying Young’s inequality {$ab \le \frac{\delta_3}{p} a^p
	+ \frac{1}{q\,\delta_3^{q/p}} b^q,$} with exponents
	\[
	p=\frac{2(2-\beta)(8-3k)}{16-24k+9\beta k}, 
	\qquad 
	q=\frac{2(2-\beta)(8-3k)}{16+12k-16\beta-3\beta k},
	\]
	and  {$\delta_3=\frac{C_5 p}{2C_6 \varepsilon}$}, we deduce
	\[
	\varepsilon^{\frac{9k}{2(8-3k)}} \|\omega\|_{L_2(D_1)}^{\frac{16-24k+9\beta k}{2(8-3k)}}
	\leq \frac{C_5}{2C_6\varepsilon}\|\omega\|^{2-\beta}_{L_2(D_1)}
	+ C_8 \varepsilon^{A(k,\beta)},
	\]
	with
	\[
	A(k,\beta) = \frac{{16-6k}}{16+12k -16\beta - 3\beta k}.
	\]
Thus, it follows from \eqref{St 30-new} that
	\begin{equation} \label{St 31-new}
		\frac{1}{2}\frac{d}{dt}\|\omega\|^2_{L_2(D)}
		+ C_5\left( \|\nabla \omega\|_{L_2(D)}^2 + \frac{1}{2\varepsilon}\|\omega\|^{2-\beta}_{L_2(D_1)} \right)
		\leq C_6 \|\omega\|_{L_2(D)}^2 + C_8 \varepsilon^{A(k,\beta)}.
	\end{equation}
	
Applying Grönwall’s lemma to \eqref{St 31-new} gives
	\[
	\|\omega(t)\|^2_{L_2(D)} \le C \varepsilon^{A(k,\beta)}, 
	\qquad \forall t \in [0,T].
	\]
Finally, integrating \eqref{St 31-new} over $[0,T]$ yields
	\[
	\|\nabla \omega\|^2_{L_2(0,T;L_2(D))} \le C \varepsilon^{A(k,\beta)}.
	\]
Since $\omega = v^\varepsilon - v$, the desired estimate \eqref{St 26} follows. 
\end{proof}

\section{Unfitted finite element method} \label{S6}

\label{sec.method}

To illustrate the performance of the method, we consider an unfitted finite element method that utilizes a background mesh in the ambient domain $D$ in a way agnostic to the geometry of the physical domain $\Omega$. In our study, we restrict the experiments to the two-dimensional case.  
\medskip

\newcommand{\Tht}{\mathcal{T}_h}
\newcommand{\restr}[2]{{\left.\kern-\nulldelimiterspace#1\right|_{#2}}}
\newcommand{\diver}{\operatorname{div}}	

Let $\Tht$ be a simplicial, shape-regular, and quasi-uniform triangulation of the ambient domain $D$, and let $\widetilde{\Tht}$ denote the mesh obtained from $\Tht$ by the Alfred split (barycentric refinement).  
On $\Tht$, we consider the inf-sup stable~\cite{arnold1992quadratic} Scott--Vogelius finite element pair
\begin{align*}
	V_h &= \{v\in H^1_0(D) \mid \restr{v}{K}\in\mathbb{P}^k(K)~\forall K\in\Tht  \}^d,\\
	Q_h &= \{q\in L^2(D) \mid \restr{q}{K}\in\mathbb{P}^{k-1}(K)~\forall K\in\Tht  \},
\end{align*}
with $k=2$. Since $\operatorname{div}(V_h)\subset Q_h$, the recovered solution is pointwise divergence-free. Therefore, we obtain a stable fictitious domain method that is strongly mass-conserving.  

For the discretization, we define the following forms:
\begin{align*}
	a(u, v) &= \mu \int_{D}\nabla u :\nabla v \,dx,\\
	b(p, v) &=  -\int_{D} p\, \diver v\,dx,\\
	c(u, v, w) &= \int_{D} (u\cdot\nabla v)\cdot w\,dx,\\
	s(u, v, w) &= \int_{D} \frac{\xi(x)v\cdot w}{\varepsilon\big(\|u\|_{L_2(D_1)}^{\beta}+{\delta_{\rm reg}}\big)} \,dx,
\end{align*}
where {$\delta_{\rm reg}>0$} is a small regularization parameter, {which is required by the method to start from the zero initial state}.  {In our experiments we set $\delta_{\rm reg}=10^{-9}$ and we found that the method is not sensitive to the  choice of $\delta_{\rm reg}$ once it is sufficiently small.}

We consider a uniform time grid $t_n = nT/N$, $n = 0,\dots,N$, where $N$ is the total number of time steps. We adopt the standard notation $f^n$ for the approximation of a quantity $f(t)$ at time $t=t_n$.  
The fully discrete variational formulation with a BDF2 semi-implicit time discretization is given by:  
Find $(v_h^n, p_h^n)\in V_h\times Q_h$, $n=1,\dots,N$, such that
\begin{multline}\label{eqn.discrete.method:TH}
	\int_{D}\frac{3 u^n_h - 4u_h^{n-1} + u_h^{n-2}}{2 \Delta t}  \phi_h \,dx 
	+ a( v^n_h, \phi_h) 
	+ c_h(2v^{n-1}_h - v^{n-2}_h,  v^n_h,  \phi_h)\\ 
	+ b_h(p_h^n,  \phi_h) + b_h(q_h, v^n_h)
	+ s(2v^{n-1}_h - v^{n-2}_h, v^n_h, \phi_h)  = \int_{D}f^\varepsilon  \phi_h \,dx
\end{multline}
for all $(\phi_h, q_h)\in V_h\times Q_h$, with the initial condition $u^0_h = v^\varepsilon(0)$.  
At each time step, a \textit{linear} problem must be solved.

A non-standard implementation issue arises when integrating the bilinear form $s$ over elements $T\in\Tht$ such that $T\cap S\neq \emptyset$, because $\xi$ is discontinuous in such elements. We address this by representing $\Omega$ as the sublevel set of a smooth level-set function $\phi$, i.e., $\xi = \frac12(\operatorname{sign}(\phi)+1)$, and by applying isoparametric unfitted finite elements~\cite{Lehrenfeld2016}.  

We perform experiments with $\Omega$ being a circle of radius $1/\sqrt{2}$ and $D=[-1,1]^2$. The exact solution is given by the velocity
\[
\mathbf{u} =
\begin{pmatrix}
	2 \pi y \, \sin(\pi t)\, \cos(\pi r^{2}) \\[6pt]
	- 2 \pi x \, \sin(\pi t)\, \cos(\pi r^{2})
\end{pmatrix},
\qquad r^{2} = x^{2} + y^{2},
\]
and the pressure
\[
p = \sin(\pi r^{2}) - \frac{2}{\pi},
\]
for $t\in[0,1]$.  

\begin{figure} 
	\includegraphics[width=0.45\textwidth]{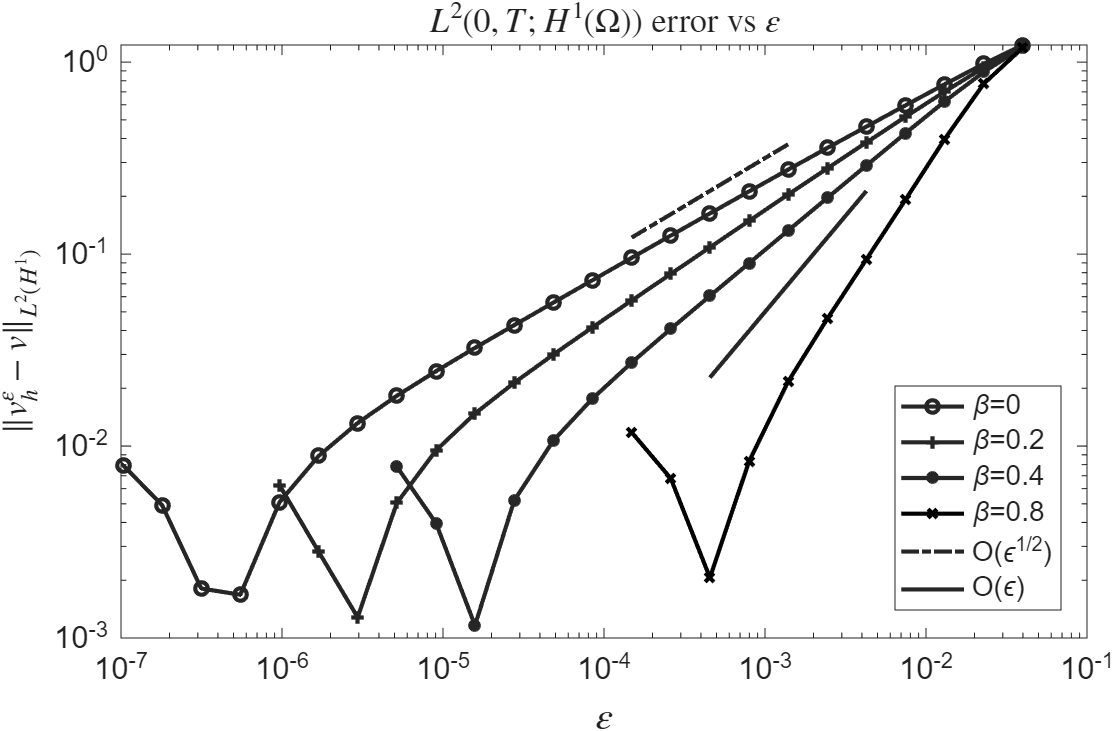}	\hfill \includegraphics[width=0.45\textwidth]{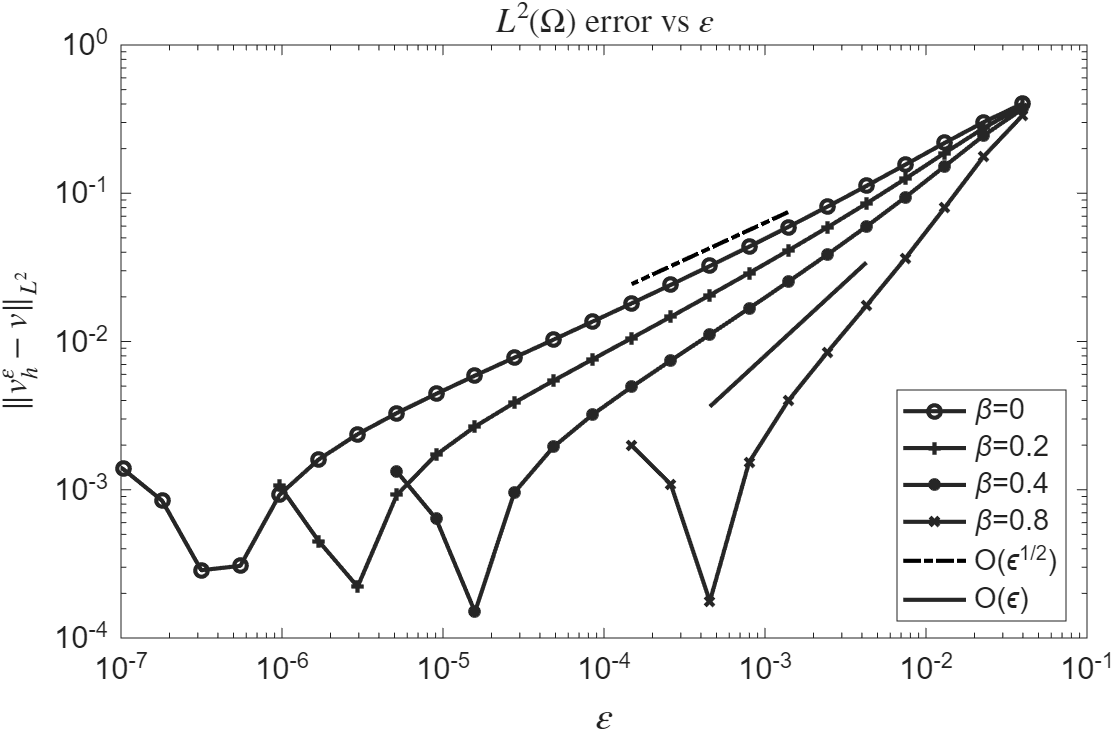}	
	\caption{The $L^2(\Omega)$-norm of the error at $t=T$ and the $L^2(0,T;H^1(\Omega))$-norm of the error versus $\varepsilon$ for various values of $\beta$. \label{Fig:1}}	
\end{figure}

For fixed discretization parameters ($h=0.0125$ and $\Delta t=0.025$), we study the convergence of the method with respect to $\varepsilon$ for several values of the exponent $\beta$. Recall that $\beta=0$ corresponds to the classical Brinkman penalization fictitious domain method. The experimental results are shown in Figure~\ref{Fig:1}, which illustrates the dependence of the $L^2(0,T;H^1(\Omega))$- and $L^2(\Omega)$-errors on $\varepsilon$.  
We observe that the achievable accuracy does not vary significantly for different values of $\beta$, with $\beta=0.4$ yielding the most accurate results among the exponents tested. We also note that for large $\beta$, the error decreases more rapidly with $\varepsilon$, in agreement with the theoretical analysis.  { For $\beta=0$ the $O(\varepsilon^{1/2})$ rate is recovered. 
For all $\beta$ tested there is a clear asymptotic convergence regime before the error drops faster  and then increases again.
This increase in error is due to the poor conditioning of the resulting algebraic systems, which is estimated to be around $10^{15}$ when the error rebounds.}
 Finally, we note that, due to the choice of finite elements, mass conservation in the physical domain is exact.

\section*{Conclusion}
We presented a simple modification of the well-known Brinkman penalization fictitious domain method for the Navier–Stokes system. For an optimal choice of the critical small penalization parameter, the classical and the modified methods achieve comparable accuracy. {The modification allows for a larger penalization parameter. Finding a nearly optimal parameter in the latter case may therefore require fewer trial computations.} The method is easy to implement and permits the direct use of strongly divergence-free finite elements, which more elaborate unfitted methods such as CutFEM find difficult to accommodate. On the theoretical side, we established well-posedness results for the fictitious domain problem and derived a new convergence estimate.

\subsection*{Acknowledgments} This research is funded by the Science Committee of the Ministry of Science and Higher Education of the Republic of Kazakhstan (grant No. AP23490027 – “Development of application software packages for the numerical solution of Navier–Stokes equations in complex domains”). The author M.O. was supported in part by the U.S. National Science Foundation under award DMS-2408978.

\small
\bibliographystyle{siam}
\bibliography{literature_updated}{}

\end{document}